\documentclass
{article}

\usepackage[T2A]{fontenc}
\usepackage{amsmath}
\usepackage{amsthm}
\usepackage{amsxtra}
\usepackage{amsfonts,amssymb}
\usepackage{latexsym}
\newtheorem{Th}{Theorem}
\newtheorem{Prop}{Proposition}

\theoremstyle{definition}
\newtheorem{Def}[Prop]{Definition}

\newtheorem{Rem}{Remark}

\begin{document}

\title{Standardness as an invariant formulation of independence}

\author{A.~M.~Vershik\footnote{This research was conducted at the Institute for Information Transmission Problems and supported by the Russian Science Foundation grant \# 14-50-00150.}}

\date{}

\maketitle

\begin{center}
To the memory of my friend mathematician Victor Havin
\end{center}

\begin{abstract}
The notion of a homogeneous standard filtration of $\sigma$-algebras was introduced by the author in 1970. The main theorem asserted that a homogeneous filtration is standard, i.e., generated by a sequence of independent random variables, if and only if the standardness criterion is satisfied. In this paper we give detailed definitions and characterizations of Markov standard filtrations. The notion of standardness is essential for applications of probabilistic, combinatorial, and algebraic nature. At the end  of the paper we present new notions of {\it shadow metric-measure space} related to nonstandard filtrations.
\end{abstract}

\section{Introduction}

A decreasing filtration (hereafter called just a filtration) is a decreasing sequence
$${\frak A}_0 \supset {\frak A}_1 \supset {\frak A}_2 \dots$$
of $\sigma$-algebras of a standard measure space
 $(X,\mu,\frak A)$ with a continuous measure~$\mu$. Here ${\frak A}_0=\frak A$ coincides with the $\sigma$-algebra of all measurable sets.

A filtration is called ergodic (or quasi-regular, Kolmogorov, satisfying the
zero--one law) if the intersection of the $\sigma$-algebras is trivial:
$$\bigcap_n {\frak A}_n
={\frak N},$$
where ${\frak N}$ is the trivial $\sigma$-algebra.

Filtrations arise in the theory of random processes (stationary or not) as the sequences of ``pasts''; in the theory of dynamical systems, as filtrations generated by orbits of periodic approximations of group actions; in statistical physics, as filtrations of families of configurations coinciding outside some volume; and, finally and most importantly, in the theory of $C^*$-algebras and combinatorics, as  tail filtrations of path spaces of Bratteli diagrams (= $\mathbb N$-graded locally finite graphs). The problem of  classification of filtrations in the category of measure spaces or other categories is deep and quite topical.
It is mainly from the classification point  of view that the subclass of filtrations we suggest below and call standard is of interest for applications. The definition of this class is  invariant  under measure-preserving transformations, and it includes the simplest case of Bernoulli filtrations, i.e., filtrations generated by sequences of independent random variables. On the other hand, this class is far from exhausting the set of all ergodic filtrations, and one may say that it opens up the stratification of the set of ergodic filtrations.

\section{Filtrations and metrics}

\subsection{Markov filtrations}

For simplicity and the economy of space, we will deal only with
{\it Markov filtrations}, but the definitions and methods in the form given below apply to the class of all filtrations, with some purely technical complications. The limitation to the class of Markov filtrations is justified also by the fact that all filtrations arising as tail filtrations in the theory of Bratteli diagrams belong to this class, so it already covers a great number of applications.

Consider an arbitrary, in general inhomogeneous in time, one-sided Markov chain
$\{y_n\}_{n\geq 0}$ with arbitrary Borel state spaces which may depend on $n$. Given  $n>0$, let
 ${\frak A}_n$ be the $\sigma$-algebra  of all measurable subsets of the space of realizations of the Markov chain that can be described in terms of random variables with indices greater than or equal to
$n$. The obtained filtration will be called the filtration generated by the Markov chain (or, sometimes, the ``tail filtration''); below we give a description of such filtrations in invariant terms.\footnote{Sometimes it is convenient to label moments of time by negative integers: $\{\dots, x_n,x_{n+1}, \dots, x_{-1}, x_0\}$; this corresponds to considering the filtration of ``pasts''
$\{{\frak A}_n\}_{n\geq 0} $
where  ${\frak A}_n$, for  $n<0$,  is the $\sigma$-algebra containing all measurable subsets of the space of realizations of the Markov chain that can be described by conditions on random variables with indices less than or equal to
$n$. However, in this paper it is convenient for us to label moments of time by nonnegative integers. Depending on the choice of the orientation, one may speak about transition or cotransition probabilities, and about exit or entrance boundaries.}

Recall that any $\sigma$-algebra is determined by a measurable partition of the measure space, which is in turn determined by a system of conditional measures (canonical system of measures in the sense of Rokhlin):
the system of conditional measures determines the partition, and hence the $\sigma$-algebra, up to isomorphism.
Thus a filtration of $\sigma$-algebras gives rise to an infinite decreasing sequence of measurable partitions
 $\{\xi_n; n\leq 0\}$, which we will use in what follows.

The partition $\xi_0$ corresponding to the $\sigma$-algebra
 ${\frak A}_0$ is the partition into singletons.

A filtration can also be given by a decreasing sequence of subalgebras of
 $L^{\infty}$ type in the space $L^{\infty}_{\mu}(X)$. A filtration determines a limiting equivalence relation on the measure space (i.e., in general, a nonmeasurable partition) and gives rise in a canonical way to a von Neumann algebra, but here we will not discuss these relations.

In what follows we assume that almost all elements of all partitions  $\xi_n$, $n\leq 0$, are finite, and thus they are finite spaces equipped with (conditional) measures; hence the conditional measure on every element of the partition is determined by a finite-dimensional probability vector.

If all conditional measures on almost all elements of the partition  $\xi_n$ coincide and are uniform, with the number of points in the elements equal to $r_n$, then the filtration is called {\it homogeneous $r_n$-adic} (in particular, {\it dyadic} if  $r_n=2^n$); if the conditional measures are uniform, but the number of points in different elements can be different, then the filtration is called {\it semihomogeneous}; this is the most interesting and important case. It corresponds to so-called central measures on  path spaces of Bratteli diagrams. But the case of dyadic filtrations already contains all difficulties of the general theory. A specific case is the study of continuous filtrations, for which all conditional measures of all quotient partitions
 $\xi_n/\xi_{n-1}$, $n=1,2, \dots$, are continuous; here we do not consider this case, but the main methods described below apply to it, too.

Let us impose an additional finiteness condition:   {\it  a filtration is said to be of finite type if not only the elements of all partitions are finite, but also the collections of the probability vectors of the conditional measures corresponding to the elements of the partition $\xi_n$ for every
 $n$ are finite.} In other words, the collection of all elements of the partition $\xi_n$  can be divided into finitely many subsets so that in each subset the vectors of conditional measures coincide. By the very definition, the class of such filtrations is invariant under measure-preserving transformations.

The following assertion holds.

\begin{Prop}
An arbitrary  finite type filtration is isomorphic to a filtration corresponding to a Markov chain with finite state sets.
\end{Prop}

The Markov filtration corresponding to a Markov process with finite state sets is, obviously, a  finite type filtration. Conversely, an arbitrary  finite type filtration can be realized as a filtration corresponding to a Markov chain with finite space sets, by choosing these sets by recursion on $n$, starting from the subsets mentioned in the definition of finiteness and subdividing these subsets if necessary to obtain a basis of the original space. Note that the filtration generated by a Markov chain with arbitrary state sets can be isomorphic to the filtration generated by a chain with finite state sets.

In what follows, we deal with  finite type filtrations generated by Markov chains.

\subsection{Transferring a metric}

We proceed to describe the main technical tool that allows one to introduce into the problem some analytic characteristics of filtrations. It is not related to the previous definitions and may be useful in various situations.

Consider an arbitrary Markov chain $\{x_n\}_{n \geq 0}$ with finite state sets and denote by $X$ the space of all its realizations, which is a general Markov compactum (in particular, it may be nonstationary). Denote by $\nu$
the Markov measure on $X$ and consider the filtration $\{{\frak A}\}_{n>0}$ of (``future'') $\sigma$-algebras. We will study metrics and semimetrics on the space $X$ that agree with the topology. A semimetric $\rho$  on $X$ is called a cylinder semimetric if there is $n$ such that $\rho(y,z)$ depends only on coordinates $y_k, z_k$ with
$k<n$. A metric $\rho$ is called an almost cylinder metric if it is the limit of cylinder semimetrics:
 $\rho(y,z)=\lim \rho_n(y,z)$. Clearly, the topology determined by an almost cylinder metric coincides with the topology of the compact space $Y$.\footnote{In the general case of a filtration in a measure space, one should use the notion of an {\it admissible metric on a measure space} introduced by the author and studied in \cite{VPZ}.}

We fix an almost cylinder metric $\rho$ and use the Markov filtration to define a sequence of semimetrics
 $\rho_0=\rho,\rho_1,\rho_2, \dots $ on the space $X$. The semimetric $\rho_n$  is constructed from the semimetric $\rho_{n-1}$ according to one and the same rule described below and called {\it transferring a metric}.

\begin{Def} Let $(X,\rho,\mu)$ be a (semi)metric measure space and $\xi$ be a measurable partition of $X$. Then on the quotient space
 $(X_{\xi},\mu_{\xi})$  there is a canonically defined semimetric $\rho_{\xi}$:
\begin{equation}\label{(Proj)}
\rho_{\xi}(x,y)\equiv K_{\rho}(\mu^x,\mu^y),
\end{equation}
where $x,y \in X$ and $\mu^x,\mu^y$ are the conditional measures of the partition $\xi$ on the elements containing the points $x,y$, respectively.
\end{Def}

Here $K_{\rho}(\cdot,\cdot)$ is the Kantorovich metric on the simplex of measures on the metric space
 $(X,\rho)$; recall the definition of this metric: the distance between two probability measures $\alpha_1,\alpha_2$ on a metric compact space $(Z,r)$ is defined as
\begin{equation}\label{K}
   K_r(\alpha_1,\alpha_2)=\inf_{\psi \in \Psi}\int_Z\int_Z
  r(a,b)d\psi(a,b),
\end{equation}
where  $\psi \in \Psi$ runs over the set of all measures on the space $Z\times Z$ with marginal (= coordinate) projections  $\alpha_1,\alpha_2$. One often says that measures $\psi$ are  ``couplings'' for the measures $\alpha_1,\alpha_2$.
Thus $\Psi$ is the set of all couplings.

Note that if $\rho$ is a metric rather than a semimetric, then
$K_{\rho}$ is a metric, too.

Observe two important properties of the operation that associates with a metric space the simplex of probability measures on this space equipped with the Kantorovich metric:

1) monotonicity proved in \cite{FA}: the inequality $\rho \leq k\rho'$ implies $K_{\rho}\leq rK_{\rho'}$;

2) linearity in the metric: $K_{a\rho_1+ b\rho_2}=aK_{\rho_1}+bK_{\rho_2}$, $a,b > 0$.

Let us return to our construction.  Successively apply the operation of transferring a semimetric to the spaces $(X/\xi_k\equiv
 X_{\xi_k},\rho_{k},\mu_{\xi_k})$ and partitions
 $\xi_{k+1}/\xi_k$, obtaining metrics $\rho_{k+1}$ on the spaces $X_{\xi_{k+1}}$,  $k=0,1,\dots$, which can be regarded as semimetrics on the original space $(X,\mu)$.

\section{Standardness: definition, standardness criterion, formulation in terms of martingales}

\subsection{Definition of standardness}

Our main definition is as follows.

 \begin{Def}
 A Markov chain is called standard if it is ergodic and for some initial metric
 $\rho=\rho_0$ on the Markov compactum $(X, \mu)$, the sequence of semimetrics $\rho_n$, $n>0$, satisfies the condition
 \begin{equation}\label{(S)}
  \lim_{n\to\infty}\int_X\int_X \rho_n(x,y)d\mu(x)\mu(y) = 0;
\end{equation}
in other words, the sequence of semimetric measure spaces
 $(X,\rho_n)$ collapses to the single-point measure space as $n \to \infty$.
 \end{Def}

The above monotonicity property of the operation of transferring a metric and the continuity of this operation with respect to the pointwise convergence of metrics immediately imply that if the standardness condition is satisfied for a given initial metric, then it is satisfied for any initial metric. Indeed, it follows from the monotonicity and linearity  that the condition holds for any cylinder semimetric, and then one should use the fact that any almost cylinder metric is a limit of cylinder semimetrics.\footnote{The same argument allows one to infer that a filtration is standard with respect to an arbitrary admissible semimetric if it is standard with respect to a single admissible metric.}

In other words, the following theorem holds.

\begin{Th}
The property of being a standard Markov filtration is invariant under the group of all measure-preserving transformations and  does not depend on the choice of the initial metric.
\end{Th}

In other words, if a filtration is realized as the tail filtration of different Markov chains, then all these chains are standard or not simultaneously.

It should be noted that computing the iterated metric is not an easy task. But it is clear from above that there exist filtrations for which the standardness condition fails for some functions, and hence there exist nonstandard filtrations. The first example of a nonstandard filtration was suggested by the author in  \cite{70} (see also \cite{73, 93}): the filtration of pasts for a random walk over trajectories of a Bernoulli action of a free non-Abelian group. By now there are many such examples. A survey of the  state of the art in this field will be published in a separate paper.

The drawback of the above definition of standardness is that checking condition~(1) requires to compute the iterated metrics. It is desirable to have a criterion that would relate the metric $\rho_n$ directly to the initial metric, skipping the intermediate steps. We present such a criterion; its statement does not involve a Markov realization of a filtration, but whenever necessary we will use special couplings, in contrast to arbitrary couplings used in the definition of the Kantorovich metric. Another difference is that the condition must be satisfied not for one metric, but for all degenerate semimetrics of a special form.

   \subsection{The standardness criterion}

Given a measurable function $f \in L^{\infty}(X,\mu)$, consider the admissible semimetric
   $\rho_f(x,y)\equiv |f(x)-f(y)|$. The standardness criterion formulated below, as well as its reformulation in terms of martingales, is just a decoding of the standardness condition for initial semimetrics of the form  $\rho_f$.

First observe the following fact.

\begin{Prop}
If the standardness condition holds for every semimetric of the form
$\rho_f(x,y)=|f(x)-f(y)|$ with $f \in L^{\infty}(X,\mu)$, then it holds for every metric.
\end{Prop}

Indeed, every metric can be approximated by linear combinations of semimetrics of the form $\rho_f$.

Consider an arbitrary  finite type filtration
 $\{\xi_n\}_n$ and the partition $\xi_n$. For what follows it is convenient to introduce a combinatorial structure on its elements. Almost every element $C \in \xi_n$ is equipped with a hierarchy of partitions obtained by restricting to it the previous partitions $\xi_1,\xi_2, \dots, \xi_{n-1}$.  Thus the finite set $C$ has a structure of a finite tree of degree $n$: the root of the tree is the element $C$ itself; the vertices of the $n$th level are the points of the set $C$; subtrees correspond to elements of the partitions $\xi_k$, $k<n$, constituting the element $C$; and the maximal branches of the tree are identified with the vertices of the $n$th level. On the set $C$, or on the set of maximal branches of the tree, the conditional measure
 $\mu^C$ is defined. We will denote the element $C\in \xi_n$ equipped with the conditional measure and the tree structure described above by
$\bar C$. In terms of the Young graph, for instance, the tree corresponds to a Young diagram $\lambda$, and its maximal branches, or the vertices of the last level, are the Young tableaux of  shape $\lambda$. For a dyadic filtration, all elements of the partitions are binary trees equipped with the uniform measure on the set of maximal branches.

Consider a coupling for two equipped elements $\bar C_1, \bar C_2$ of the partition $\xi_n$, i.e., a measure $\psi \in \Psi(\bar C_1, \bar C_2)$ on the direct product of trees $\bar C_1 \times \bar C_2$ supported by a tree equipped with a measure such that the projections of this tree and measure to the coordinates $\bar C_1, \bar C_2$ coincide with the corresponding structures in
$\bar C_1$  and $\bar C_2$, respectively.  Such measures will be called {\it special couplings of trees with measures}.
In our case (of finite type filtrations) all spaces $\bar C$ are finite, so that couplings are given by matrices.

Now we are ready to formulate the standardness criterion as a property of filtrations.

 \begin{Def}
A filtration $\{\xi_n\}_{n\in N}$ in a space $(X,\mu)$ satisfies the standardness criterion if for any measurable function
 $f \in L^{\infty}(X,\mu)$ and any
 $\varepsilon >0$ there exists $N$ such that for all $n>N$ the following inequality holds:
 \begin{multline}\label{4}(\mu\times \mu) \big\{(x_1,x_2) \in X \times X:\\
 \int_{C^{x_1}\times C^{x_2}}
 |f^{\bar C^{x_1}}(z_1) - f^{\bar C^{x_2}}(z_2)|\,d\psi^{C_1,C_2}(z_1,z_2) <\varepsilon\big\}>1-\varepsilon.
 \end{multline}
 \end{Def}

Let us make clear that $C^{x_i}$ is the element of the partition $\xi_n$ that contains the point $x_i$, $i=1,2$, and $\psi^{C_1,C_2} $ is a special coupling of these two elements. It is not needed to take the minimum over all special couplings, since the existence of only one coupling with the desired property is sufficient. Informally speaking, condition~(4) means that for sufficiently large $n$ the restrictions of an arbitrary function from
 $L^{\infty}$ to most elements of the partition $\xi_n$ coincide up to $\varepsilon$ and to a coupling that agrees with the conditional measures and tree structures on these elements.

It suffices to verify the criterion only for functions $f$ with finitely many values, which turns checking the criterion into a purely combinatorial problem. On the other hand, it suffices to verify the criterion for a single one-to-one function, but then the complexity of the criterion will be comparable with that of the original standardness condition for metrics.

We emphasize that the conditions on special couplings between partition elements in the above criterion is much stronger than the condition on couplings in the definition of the Kantorovich metric, since a coupling is required not only to preserve the measure, but also to preserve the tree structures. Hence one could call them ``Markov'' couplings.

The statement of the standardness criterion implies the following important property.

\begin{Rem}
If the criterion is satisfied for a filtration  $\{\xi_n\}_{n>0}$, then it is satisfied for the quotient filtration
$\{\xi_{n+k}/\xi_k\}_{n>0}$ for every $k>0$.
\end{Rem}

Finally, we emphasize that the criterion condition  is invariant by the very definition, i.e., if it holds for a filtration, then it also holds for any isomorphic filtration. Indeed, it involves only notions related to the filtration and no other notions (e.g., metrics, as in the first statement). Although we formulated the criterion for  discrete type filtrations, it applies with minimal modifications to arbitrary filtrations.

\begin{Th}
A Markov filtration is standard if and only if it satisfies the standardness criterion.
 \end{Th}

\begin{proof}
Essentially, we must verify that for semimetrics of the form $\rho_f$ where $f$ is a function from $L^{\infty}$, the criterion is a decoding of condition~(3). In other words, condition~(4) is a concise formula describing all successive transfers of the semimetric along the partitions of the filtration: the condition on couplings of  trees with measures appeared exactly in this way.
\end{proof}

Let us specify the standardness criterion for homogeneous filtrations and functions with finitely many values. For clarity, we restrict ourselves to dyadic filtrations and indicator functions of sets. What does the criterion mean in this case?

In this case, an element of the $n$th partition is a set consisting of $2^n$ points equipped with the uniform measure and a structure of a binary tree (of height~$n$), and the restriction
of an indicator function is a $0-1$ vector of dimension~$2^n$. In the case of a uniform measure,   as couplings  we may take not Markov,
but one-to-one maps
(which preserve the uniform measure), that is, elements of the group preserving the tree structure, i.e., automoprhisms of the binary tree. Hence the distance between the restrictions of a function to two trees is the distance between two orbits of the group of automorphisms of the tree acting on the vertices of the unit cube of dimension
 $2^n$. Thus the standardness criterion means that for every $\varepsilon>0$ there is $N$ such that for all elements of the partition $\xi_n$ with $n>N$ from some set of measure $>1-\varepsilon$, the restrictions of the indicator function lie on orbits of the action of the group $D_n$ of automoprhisms of the tree for which the (Hamming) distance is less than
$\varepsilon$.
This observation can easily be extended to the case of other filtrations.

The meaning of the standardness criterion is that the restrictions of any function to various elements of the partition have asymptotically  the same behavior with respect to the tree structure (up to automorphisms, or couplings). {\it Hence the standardness can be interpreted as the asymptotic Bernoulli property for filtrations, i.e., as an analog of the independence of a sequence of random variables.}

Historically, the statement of the standardness criterion for homogeneous filtrations preceded the standardness condition~(3) given above. The fact that if the criterion is satisfied then the filtration is standard (in the homogeneous case, Bernoulli) was proved in 1970 by the author (\cite{70, 73},
see \cite{93}), who simultaneously gave the first example of a nonstandard filtration.

\subsection{Characterization of standard filtrations in terms of martingales}

Now we give yet another formulation of the standardness criterion, in terms closer to the theory of random processes. Here we use the terminilogy and properties of Markov processes.

Consider a sequence of scalar (e.g., real) random variables $f_n$, $n>0 $, that constitute a Markov chain with finite sets of transitions (but with arbitrary state sets), in general nonstationary, and the filtration generated by this chain: $\{{\frak A}_n; {\frak A}_n= \ll f_k, k \geq n\gg, \; n=1,2, \dots \}$. \
Assume that the zero--one law holds, i.e., the filtration is ergodic.

We will rephrase the standardness criterion for this filtration assuming that all  $f_n$ have finite first moments, and explain in what sense the criterion is a strengthening of Doob's martingale convergence theorem. The latter  says, for instance, that for all $k$
$$ \lim_n E[f_k|{\frak A}_n] = Ef_k$$
(the almost everywhere convergence of conditional expectations).

The same theorem can be applied not to the random variables themselves, but to their conditional distributions: as $n \to \infty$, the conditional distribution of the random variable $f_1$ (or of several first variables $f_1,\dots, f_k$) with respect to the $\sigma$-algebra ${\frak A}_n$, $n>k$, converges almost everywhere to the unconditional distribution. This fact uses only the ergodicity of the filtration (the zero--one law). For our purposes it is convenient to eliminate the limiting (unconditional) distribution from this statement and rephrase the theorem as follows: the distance between the conditional distributions given $f_n=x_n$ and $f_n=y_n$ tends to zero as
$n\to\infty$ for almost all pairs $(x_n,y_n)$ of trajectories of the Markov chain.

The following assertion is a maximal (``diagonal'') strengthening of these theorems in which $k$ is not fixed but equal to $n$, i.e., tends to infinity simultaneously with the number of the $\sigma$-algebra with respect to which the conditional expectation is taken. This condition does not hold for all Markov chains, but only for standard filtrations (standard Markov chains).
Let us  formulate it precisely.

Consider the conditional distributions $\mu^x_n$ and $\mu^y_n$ of the first $n-1$ random variables
$ \{f_1,f_2, \dots , f_n\}$ given $f_{n+1}=x$
and $f_{n+1}=y$; the Markov property means that these conditional distributions are discrete measures on the set $\mathbb R^n$ of vectors   (trajectories of the chain)
which depend only on the values $x$ and $y$.  Take a separable metric $\rho$ in the space of all trajectories of the process, and approximate it by a metric on the set of these vectors from $\mathbb R^n$. Then consider the value  $$r_n(x,y)=\min_{\psi \in \Psi_m} E K_{\rho}(\mu^x_n, \mu^y_n),$$
where the minimum is taken over all Markov couplings, as in the standardness criterion. We emphasize that the Markov condition imposed on couplings (i.e., the requirement that the projections to the coordinates should preserve not only the measure, but also the order structure) is of crucial importance.

The above argument imply the following theorem.

 \begin{Th}
 A Markov chain is standard if and only if
  $$\lim_n\int_{X\times X}r_n(x,y)d\mu(x)d\mu(y)=0.$$
 \end{Th}

Now we will relate this statement to limit shape theorems. Regarding trajectories of the Markov chain as paths in the Bratteli diagram, the standardness of a central measure on paths can be formulated as follows:
{\it for every $\epsilon>0$ there is $N$ such that for every $n>N$ there exists a vertex $v_n$ of the $n$th level of the graph such that the measure of the set of paths that meet this level at vertices from the
$\epsilon$-neighborhood of  $v_n$ is not less than $1-\epsilon$}. Here a neighborhood is understood in the sense of the $n$-iterated arbitrary initial metric. In examples where vertices of the graph are some sort of configurations, this fact turns into a theorem on the concentration of a random distribution near some configuration.

\section{Isomorphism of standard finitely isomorphic filtrations}

\subsection{Finite invariants and isomorphism}

Above we have defined the notion of a standard Markov chain and a standard Markov filtration. Let us see how standard filtrations correlate with all the other ones.

\begin{Def}
Two filtrations $\{{\frak A}\}_n$ and $\{{\frak A'}\}_n$ of measure spaces $(X,\mu,\frak A)$ and $(X',\mu',\frak A')$ are called finitely isomorphic if for every  $n$ there exists an automorphism of the space $(X,\mu)$ that sends ${\frak A}_k$, $k=1, \dots, n$, to ${\frak A'}_k$, $k=1, \dots, n$.
\end{Def}

It is not difficult to describe finite invariants of semihomogeneous filtrations (recall that these are filtrations for which the conditional measures of almost all elements of all partitions $\xi_n$ are uniform\footnote{The conditional measures of the partitions
$\xi_n/\xi_{n-1}$ are not necessarily uniform in this case.}). These are collections, for each $n$, of the measures of the unions of  all elements of the partition
$\xi_n$ that have isomorphic structures of the restrictions of the previous partitions. For
$n=1$, they are just the measures of the unions of all elements of $\xi_1$
 of cardinality 1, 2, etc.; for $n=2$, these are the measures of the unions of all elements of $\xi_2$ that are comprised of elements of $\xi_1$ with given  cardinalities; and so on. This collection of numbers is finite for every $n$ (for semihomogeneous finite type filtrations) and subject to no restrictions. It is convenient to encode finite invariants of finite type filtrations by a multigraph (generalized Bratteli diagram) with a fixed central measure on the set of its paths. An equipped graph allows one to model not only semihomogeneous filtrations, but arbitrary finite type filtrations.

The standardness imposes some restrictions on these collections of numbers, and in this section we consider only those finite isomorphism classes for which a standard partition does exist. The problem of finding estimates on the growth of the sequence of finite invariants that would guarantee the existence of standard filtrations is of great interest.

The intersection of the $\sigma$-algebras is not an invariant of finite isomorphism classes, so it makes sense to consider finite isomorphism for ergodic filtrations. On the other hand, as we know, finite isomorphism does not imply isomorphism even for ergodic filtrations. However, the following important fact holds.

 \begin{Th}
 Two finitely isomorphic standard ergodic Markov filtrations are metrically isomorphic.
\end{Th}

Thus a wider space of ergodic Markov filtrations is foliated over the space of standard filtrations.
The study of a fiber, i.e., of all filtrations finitely isomorphic, e.g., to Bernoulli filtrations, is an interesting and difficult problem.

\begin{proof}
The argument largely resembles the proof of the sufficiency part of the standardness criterion in the homogeneous case, however here we do not construct a normal form of a filtration (in the homogeneous case, a Bernoulli one), but directly prove the isomorphism of standard finitely isomorphic filtrations.

Consider two such filtrations $\{\xi_n\}$ and $\{\xi'_n\}$. Without loss of generality we may assume that both are defined on the same measure space. We will construct bases of the full $\sigma$-algebras of measurable sets rigidly associated with each  filtration. A basis is a sequence $\{\eta_n\}_n$ of finite partitions whose product is the partition
$\epsilon$ into singletons: $\bigvee \eta_n = \epsilon$. Also choose a sequence of positive numbers
$\alpha_n$, $n=1,2,\dots$, that tends to zero. To construct bases, we will use the standardness criterion, starting from some arbitrary total countable system of functions $f_n \in L^{\infty}$, $n=1,2, \dots $. Apply the criterion to the function
 $f_1$,  with $\alpha_1$ as the number $\varepsilon>0$ from its statement. For each filtration there is $n$ (take the greater of these $n$ and call it $n_1$) such that inequality~(4) holds for  the function $f_1$ and elements of the partition $\xi_n$
constituting  a set of measure $>1-\alpha_1$.
 This means that this function is uniquely determined by the structure of partition elements  up to
  $\alpha_1$ (in the $L^1$ norm). But the filtrations are finitely isomorphic, hence there is an isomorphism of the base space that sends the first $n$ partitions of the second filtration to the first $n$ partitions of the first filtration, and thus sends $f_1$ to a function that differs from it by at most
$\alpha_1$. Passing to the quotient space $X/\xi_{{n_1}}$, considering the projection of
$\{f_n\}$, $n>1$, to  $X/\xi_{{n_1}}$, and using the fact that the standardness criterion and finite isomorphism remain valid for quotient filtrations, we repeat the previous argument with necessary modifications ($\alpha_2$ instead of $\alpha_1$, $n_2$ instead of $n_1$, etc.). The new isomorphism of quotient filtrations can be lifted to an isomorphism in the original space that sends a fragment of the second filtration to a fragment of the first filtration up to the partition
 $\xi_{n_2}$. Continuing this process, we obtain a sequence of isomorphisms. The crucial fact is that this sequence converges, which is guaranteed by the totality of the system of functions  $f_n$.\footnote{Note that it is the absence of such a convergence that underlies the existence of non-isomorphic filtrations.} Thus we have constructed an isomorphism of two finitely isomorphic standard filtrations.
 \end{proof}

\begin{Rem}
 1. The  theorem just proved is useful when considering the important class of {\it stationary filtrations and "asymptotically stationary" ones}, which arise as the filtrations of ``pasts'' of stationary random processes, as well as  in the theory of graded graphs. The study of this class is closely related to the isomorphism problem in ergodic theory.

2. The relation in the class of all filtrations determined by the isomorphism (and finite isomorphism) of filtrations is too restrictive, since it breaks down when one changes finite fragments of filtrations. A more acceptable relation is the asymptotic isomorphism, the isomorphism of the ``tails'' of filtrations. In the stationary case, it coincides with the global isomorphism. The standardness is evidently invariant of asymptotic isomorphism. Other asymptotic invariants will be considered in the prepared survey on filtration theory (see the last paragraph of this article).
 \end{Rem}

The modern theory of filtrations started from the lacunary isomorphism theorem for dyadic sequences
\cite{68}. An almost verbatim repetition of its proof yields the following general lacunary isomorphism theorem.

 \begin{Th}
Every ergodic finite type filtration $\{\xi_n\}_{n=1}^{\infty}$ is lacunary isomorphic to a standard filtration, i.e., there exists a sequence of numbers $n_k \to \infty$ such that the filtration $\{\xi_{n_k}\}_{k=1}^{\infty}$ is standard.
 \end{Th}

Remark that the filtration in the theorem could be not finitely isomorphic to any standard filtration.

 There exists unique up to {\it finite isomorphism}
filtration of the continuous type: the conditional measures almost all elements of the partition $\xi_n/\xi_{n-1}, n=1 \dots$ are continuous. The standard filtration at this class is Bernoulli filtration with continuous state space. The invariants with respect to isomorphism are important for the theory of stochastic processes.

\subsection{Examples}

1. A standard homogeneous filtration is Bernoulli. The corresponding Bratteli diagram is the diagram of an infinite tensor product. It is not difficult to give an explicit formula for standard semihomogeneous filtrations using graded graphs.

2. An interesting and simple example of finitely isomorphic but nonisomorphic  ergodic Markov stationary filtrations is as follows:
$$X=\prod_1^{\infty} \{0;1\},$$
with the transition matrices
$$\left(
  \begin{array}{cc}
    p & q \\
    p & q \\
  \end{array}
\right)\qquad\text{and}\qquad\left(
  \begin{array}{cc}
    p & q \\
    q & p \\
  \end{array}
\right),$$
where $p \ne q$. The first filtration is Bernoulli, and hence standard. The second one is not standard, which can easily be seen by checking the violation of the criterion condition: namely, the distance between two distinct vertices is equal to
 $|p-q|$ at all levels (see also below). If a stationary Markov chain with finitely many states generates a homogeneous filtration, then the standardness holds.

3. A typical example of a nonstandard Markov homogeneous filtration is as follows. Consider a Bernoulli action of an infinite countable group and the Markov process of walking over trajectories of this action determined by a measure on the group. The first example of a nonstandard dyadic filtration involved a simple walk over trajectories of a Bernoulli action of the free group with two generators (see \cite{70}).\footnote{In 1968, when the author started  working on these problems, he believed that in the dyadic case all ergodic filtrations are standard and mentioned this in a short remark in the paper on lacunary isomorphism \cite{68}. The same error was also made, still earlier, by N.~Wiener in the book \cite{W},  chapter ``Decoding.''}
For the group $\Bbb Z$, we obtain Kalikow's Markov shift, or the $(T,T^{-1})$ transformation; the nonstandardness of the past of this process was conjectured by the author and proved in \cite{Ru}. For the multidimensional lattice  ${\Bbb Z}^d$, this is proved in \cite{VG}. The general situation has not been studied.

4. There are fragmentary results on the standardness of the tail filtration for central measures on some Bratteli diagrams (the Pascal graph \cite{De}, the multidimensional Pascal graph, the Young graph, etc., see \cite{14}).

5. The problem of standardness or non-standardness of a stationary Markov chain (and the corresponding filtration) in the case of an infinite state set is an important open question. For random walks over trajectories of measure-preserving group actions, the answer depends on the properties of the action: for actions with discrete spectrum the standardness holds, Bernoulli actions are considered above. For general random walks in a random environment no general results are known.

\section{Nonstandardness: shadow metric-measure spaces}

We will restrict ourselves only to the main definition of a new notion,  which will be considered in a separate paper.

Assume that an ergodic filtration is not standard and the limit of the integrals in the left-hand side of (3), which, as one can easily see, always exists, does not vanish. Does this mean that the sequence of semi-metrics converges to some semi-metric?
It turns out that the sequence of semi-metrics $\rho_n$ cannot converge on a set of pairs of points of positive measure
$\mu\times \mu$, since the converse would contradict the ergodicity of the filtration. However, in some examples at least, there exists a limit of the matrix distributions of the metrics, i.e., a limit of the measures on the space of distance matrices corresponding to taking random independent samples of points of the metric space.

First recall what is the matrix distribution of a Gromov--Vershik metric triple $(X,\mu,\rho)$  (see \cite{Gr,U,2002}). Denote by ${\Bbb M}_{\Bbb N}^{dist}$ the space of distance matrices, i.e., the cone of infinite real symmetric matrices with zeros on the main diagonal and entries satisfying the triangle inequality. Consider the map
$$F_{\rho}:\{x_i\}_{i=1}^{\infty}\mapsto \{\rho(x_i,x_j)\}$$
from the infinite product $X^{\infty}$ equipped with the Bernoulli measure $\mu^{\infty}$ to the space of distance matrices. The image of the Bernoulli measure under this map is called the matrix distribution of the triple
$(X,\mu,\rho)$.  If the measure $\mu$ is not degenerate (i.e., is positive on nonempty open sets), then, by a theorem from \cite{U,Gr},
the matrix distribution is a complete invariant of the metric triple with respect to measure-preserving isometries. One can describe exactly what measures can appear as such images, i.e., what measures can be matrix distributions of metric triples, see \cite{2002}.

However, the set of matrix distributions is not weakly closed in the space of measures on the cone ${\Bbb M}_{\Bbb N}^{dist}$. Hence we can define limiting objects for which analogs of the matrix distribution are measures lying in the closure of the set of matrix distributions of metrics.

\begin{Def}
A space with a continuous measure $(X,\mu)$ together with a measure $D$
on the cone ${\Bbb M}_{\Bbb N}^{dist}$ lying in the weak closure of the set of matrix distributions
is called a shadow metric triple if there exists a sequence of admissible semi-metrics $\rho_n$ on $(X,\mu)$ such that
 $$w\text{-}\lim_n D_{\rho_n}=D,$$
where $D_{\rho_n}$ is the matrix distribution of the semi-metric $\rho_n$.
 \end{Def}

Note that the measure $D$ is invariant under the action of the infinite symmetric group $ S_{\Bbb N}$  by simultaneous permutations of rows and columns of matrices. Thus an ordinary metric triple is also a shadow one.

Speaking somewhat roughly, we define a map
$$(X^{\infty},\mu^{\infty}) \rightarrow ({\Bbb M}_{\Bbb N}^{dist}, D_{X,\mu})
$$
as the limit of the maps $F_{\rho_n}$. In other words, although the limit of the metrics
$\rho_n$ may not exist, yet the limit distribution of the collection of distances does exist.

For instance, it follows from this definition that for every pair of points of a shadow metric triple there exists a distribution on the half-line
${\Bbb R}_+$ that plays the role of a random distance and does not depend on this pair. In Example~2, say, this is the distribution  $(1/2,1/2)$ on the distances $0$ and $|p-q|$. It is the limiting (more exactly, empirical) distribution for the sequence of distances between two points. In the ordinary sense of the word, a metric and points are not defined  in a shadow mm-space. However, in such a space there is, for instance, a well-defined notion of the  $\epsilon$-entropy (the so-called ``secondary entropy''), and other invariants of metric spaces are defined.

It is not yet known whether such limiting distributions exist for any nonstandard filtration and whether such a filtration always gives rise to a shadow metric triple. Apparently, these new distributions must play a significant role in the theory of random processes and their applications. In this connection, we would like to mention, in particular, that in the stationary case the notion of standardness is similar to, but does not coincide with, D.~Ornstein's VWB condition
(\cite{O}), and nonstandardness arise in examples of non-Bernoulli systems with completely positive entropy.

Observe an outward similarity between the notions of shadow metric triples and ``pointless spaces'' in von Neumann's continuous geometries. We will not dwell on the existing relations between non-standardness and the theory of von Neumann factors or the theory of ergodic equivalence relations, as well as on the links between the notion of standardness and the theory of adic transformations and the representation theory of AF algebras.

\end{document}